\documentclass[11pt]{article}
\usepackage{amssymb}
\usepackage{amsmath}
\usepackage{amsthm}
\usepackage[all]{xy}
\usepackage{david}

\newcommand{\Sup}{\mathcal{S}up}

\DeclareMathOperator{\LMod}{\!-Mod}
\DeclareMathOperator{\RMod}{Mod-\!}
\DeclareMathOperator{\BMod}{\!-Mod-\!}

\begin{document}
\title{Quantum triads: an algebraic approach}
\author{David Kruml\thanks{Supported by the Grant Agency
of the Czech Republic under the grant No. 201/06/0664.}}
\maketitle


\begin{abstract}
A concept of quantum triad and its solution
is introduced.
It represents a common framework for several
situations where we have a quantale with a right
module and a left module, provided with
a bilinear inner product.
Examples include Van den Bossche quantaloids, quantum
frames, simple and Galois quantales,
operator algebras, or orthomodular lattices.
\end{abstract}

\section{Introduction}
The lattice of ideals of a commutative ring is
a very useful characteristic and provides to apply
topological techniques in ring theory.
When the ring is non-commutative, the two-sided ideals ($T$)
are much less descriptive
and we should rather consider lattices of right ($R$) or
left ($L$) ideals, or even better all of them.

The lattices are naturally equipped by an associative
multiplication which distributes over joins and
this led J.~C.~Mulvey to introduce a concept of
\emph{quantale}.
Quantales arising from right ideals were studied
by many authors
(see cf. \cite{BoRoVB89,CoMi01,Roma05}, etc.).
In \cite{VB95} (see also \cite{St03})
G.~Van~den~Bossche advanced an idea of
F.~W.~Lawvere to consider the lattices $L,T,R$ as hom-sets
in a two-object quantaloid together with a lattice $Q$
of all subgroups which are modules of a center:
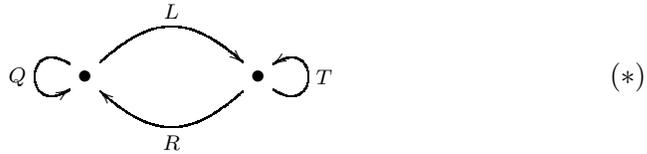
\begin{equation}\label{law}
\xymatrix{ \bullet \ar@(ur,ul)[rr]^L \ar@(ul,dl)[]_Q & &
\bullet \ar@(dr,ur)[]_T \ar@(dl,dr)[ll]^R }
\end{equation}
Notice that $R,L$ are then considered as modules
rather than quantales and that the scheme preserves
a quite important multiplication $L\times R\to T$.
This importance is well visible when we are dealing with
an operator algebra $A$ and the ideals are
realized by projections
--- for projections $p,q$ the more interesting product
$ApqA$ is obtained by multiplying left ideal $Ap$
with right ideal $qA$, while right quantale structure on $R$
gives only $pAqA$ which can be calculated as
a right action of two-sided $AqA$ on $pA$.
This fact was recognized by J.~Rosick\'y \cite{Rosi89}
and studied in a context of \emph{quantum frames}.

The aim of this paper is to formalize the
relationship among one- and two- sided ideals
and to construct a quantale subsuming the
structure. 
Dropping $Q$ from the Van den Bossche quantaloid
but preserving all remaining compositions
and associative laws we obtain a basic
example of what is called \emph{quantum triad}
$(L,T,R)$.
The quantum triads, or shortly just triads,
can be understood as multiplicative
Chu spaces but notice that morphisms
of triads (which are not studied here) would
arise from ring morphisms and this is
a different philosophy then the one of Chu spaces.

The fill-in by quantale $Q$ in (\ref{law})
is an instance of \emph{solution} of the triad $(L,T,R)$.
It is shown that every triad has
two extremal solutions denoted by $Q_0$ and $Q_1$
and they enclose a category of all solutions.
The solution $Q_0$ is realized by tensor product
$R\otimes_TL$ while $Q_1$ is a generalization
of a quantale of endomorphisms.
The results are based on a special case studied
in \cite{EgKr08} and further communication
with J.~Egger and the idea of P.~Resende \cite{Re04}
who constructed $Q_1$ for the case of Galois connections.
The quantales $Q_0,Q_1$ reflects two aspects of
the quantization of topology --- it is a non-commutative
intersection represented by multiplication on $Q_0$, and
transitivity of states represented by actions of $Q_1$
on $L$ and $R$.

We discuss properties when $L,T,R$ appears as left-, two-, and 
right-sided elements of a solution, that is when
the solution represents a unique object
covering all components of the triad.
A special attention is kept to
involutive ($L\cong R$) and Girard ($L\op\cong R$)
triads.
In particular, solutions of a triad given by a complete
orthomodular lattice represents a contribution to
topics of dynamical aspects of quantum logic \cite{CoMoWi00}.

Since one can find other examples of quantum triads
outside lattice theory, it is reasonable to work
with a maximal generality.
The author presents here only applications
in sup-lattices and uses an algebraic language.
A categorical approach will be presented in a separate
paper \cite{Kr08}.

\section{Preliminaries}
Recall that a \emph{category of sup-lattices}
consists of complete lattices as objects
and suprema preserving maps as morphisms.
Every sup-lattice morphism $f:S\to S'$ has
an \emph{adjoint} $f\adj:S'\to S$ given by
$f(x)\leq y\LRa x\leq f\adj (y)$
which preserves all infima.
By dualizing $S$ and $S'$ we obtain a
sup-lattice morphism denoted by $f^*:(S')\op\to S\op$.
A map $f:S\times S'\to S''$ of sup-lattices is
called a \emph{bimorphism} if it is a morphism
in both variables, i.e. fixing an element of
$x\in S$ (or $y\in S'$) we obtain morphism
$f(x,-):S'\to S''$ (or $f(-,y):S\to S'')$.
When the bimorphism is apparent, an element
$f(x,y)$ is understood as a products and
denoted by $xy$.
The adjoints of $f(x,-),f(-,y)$ are referred
as \emph{residuations} and denoted by $-\la x,y\ra -$.

A \emph{quantale} $Q$ is a sup-lattice equipped
by an associative bimorphism $Q\times Q\to Q$.
The top or bottom element is denoted by $1$ or
$0$, respectively.
The quantale is called

\emph{unital} if it admits a unit $e\in Q$, i.e.
$qe=q=eq$ for every $q\in Q$,

\emph{semiunital} if $q\leq q1\wedge 1q$ for
every $q\in Q$,

\emph{strictly two-sided} if it is unital
and the unit coincide with the top element,

\emph{Girard} if admits an element $d\in Q$ which is
\emph{cyclic}, i.e. $qq'\leq d\LRa q'q\leq d$,
and \emph{dualizing}, i.e.
$q=d\la (q\ra d)=(d\la q)\ra d$, for every $q,q'\in Q$.
The element $q\comp=q\ra d=d\la q$ is regarded
as a \emph{complement} of $q$.

A quantale $Q$ is called \emph{involutive}
if it is equipped by a unary operation $^*$ of
\emph{involution} provided that $(q^*)^*=q$,
$(qq')^*=(q')^*q^*$ and $(\bigvee q_i)^*=\bigvee q_i^*$
for all $q,q',q_i\in Q$.

An element $q\in Q$ is said to be \emph{right-,
left-,} or \emph{two-sided} if $q1\leq q,
1q\leq q$ or both the inequalities hold,
respectively.
The respective sup-lattices are denoted by
$\R(Q),\LL(Q),\T(Q)$.
Recall that every unital quantale is semiunital
and in a semiunital quantale it holds that
$r1=r,1l=l$ for all $r\in\R(Q),l\in\LL(Q)$.
Since $1^*=1$ in any involutive quantale,
the involution provides a sup-lattice
isomorphism between $\R(Q)$ and $\LL(Q)$.

A sup-lattice morphism $f:Q\to K$ between
quantales $Q,K$ is called a
\emph{(involutive) quantale morphism} if it
preserves the multiplication (and the involution),
i.e. $f(qq')=f(q)f(q')$ (and $f(q^*)=f(q)^*$).
The morphism $f$ is called \emph{strong}
if it preserves the top element, i.e.
$f(1_Q)=1_K$.
It follows easily that a strong morphism preserve
also right- and left-sided elements.

The two-element sup-lattice $\Two=\{ 0,1\}$,
as well as any frame, will be
regarded as a unital (involutive) quantale
with multiplication $\wedge$ (and trivial involution).

A sup-lattice $M$ is called a \emph{left $Q$-module}
if there is a bimorphism $Q\times M\to M$ associative
with the quantale multiplication, i.e.
$(qq')m=q(q'm)$ for every $q,q'\in Q,m\in M$.
$M$ is said to be unital when $Q$ is unital and
$em=m$ for every $m\in M$.
Right modules are defined in an analogous way.
$M$ is called a \emph{$(Q,Q')$-bimodule} for
quantales $Q,Q'$ if it is left $Q$-module,
right $Q'$-module and it holds that $(qm)q'=q(mq')$
for every $q\in Q,m\in M,q'\in Q'$.
Notice that every quantale $Q$ is automatically
$(Q,Q)$-bimodule.
Categories of left $Q$-modules, right $Q'$-modules
and $(Q,Q')$-bimodules are denoted by $Q\LMod,\RMod Q,Q\BMod Q'$,
respectively.
A sup-lattice morphism $f:M\to M'$ between two
left $Q$-modules is called a \emph{$Q$-module morphism}
if $f(qm)=qf(m)$ for every $q\in Q,m\in M$, etc.

Quantales $C,Q$ together with a morphism $\phi:C\to Q$
are called a \emph{couple of quantales} if
$C$ is a $(Q,Q)$-bimodule, $\phi$ is a $(Q,Q)$-bimodule
morphism and it holds that
$cc'=\phi(c)c'=c\phi(c')$ for every $c,c'\in C$.

A couple $C\stackrel{\phi}{\to}Q$ is said to be

\emph{unital} if $Q$ is unital and $C$ is a unital
$(Q,Q)$-module,

\emph{Girard} if $C$ admits a cyclic dualizing element $d$,
but now with respect to the bimodule actions, i.e.
$qc\leq d\LRa cq\leq d$,
$q=d\la (q\ra d)=(d\la q)\ra d$, where
the residuations are calculated as adjoints
of $q-,-q:C\to C$, and $c=d\la (c\ra d)=(d\la c)\ra d$
with residuations adjoint to $c-,-c:Q\to C$,
for all $q\in Q,c\in C$.

\section{Triads and solutions}

\begin{defin}
A \emph{(quantum) triad} 
consists of the following data:
\begin{itemize}
\item quantale $T$,
\item right $T$-module $R$,
\item left $T$-module $L$,
\item and bimorphism
$L\times R\to T$ compatible with the module actions,
\end{itemize}
i.e., there are four bimorphisms,
referred as (TT, RT, TL, LR), satisfying
all the five reasonable associative laws
(TTT, TTL, RTT, LRT, TLR).
\end{defin}

\begin{defin}
A quantale $Q$ is said to be a \emph{solution}
of triad $(L,T,R)$ if
\begin{itemize}
\item $R$ is a $(Q,T)$-bimodule,
\item $L$ is a $(T,Q)$-bimodule,
\item there is a compatible bimorphism $R\times L\to Q, (r,l)\mapsto rl$,
\end{itemize}
which means that  
we add further four bimorphisms
(QQ, QR, LQ, RL) and
require all the remaining associative laws
(QQQ, LQQ, QQR, TLQ, QRT, QRL, RLQ, RTL, LQR, RLR, LRL)
for scheme (\ref{law}).
\end{defin}

\begin{example}\label{ex}
(1) Let $Q$ be a quantale. Then $(\LL(Q),\T(Q),\R(Q))$
is a triad and $Q$ is a solution.

(2) Let $A$ be a ring. As mentioned in Introduction
we assume $L,T,R$ to be sup-lattices of left-,
two-, and right-sided ideals.
As solution we can consider a quantale of all additive
subgroups of $A$ or a quantale of those subgroups
that are modules over the center of $A$.

(3) When $A$ is a C*-algebra, it is possible to
consider only ideals closed in norm topology.
Then spectrum $\Max A$ consisting of all closed
subspaces \cite{MuPe01}
is a solution.

(4) Let $S$ be a sup-lattice.
Putting
$$ xy=\begin{cases}
0, & y\leq x, \\ 1, & y\nleq x
\end{cases} $$
for $x,y\in S$ we obtain a bilinear map $S\op\times S\to\Two$.
Since every sup-lattice is a unital $\Two$-module,
we get a triad $(S\op,\Two,S)$.
Quantale $\Q(S)$ of all sup-lattice endomorphisms \cite{PeRo97}
of $S$ and quantale $\C(S)=S\otimes S\op$ are clearly
solutions of the triad \cite{EgKr08}.

(5) Let $H$ be a Hilbert space.
Then left ideals of operator algebra $\B(H)$
closed in normal topology,
as well as those right ideals,
can be identified with closed subspaces of $H$.
The sup-lattice is denoted by $L(H)$.
The only closed two-sided ideals are $\{ 0\}$ and $A$,
hence we obtain a triad $(L(H),\Two,L(H))$
which is a special case of (4).
Except $\Q(L(H)),\C(L(H))$ there are also solutions
$\Max\B(H), \Max_{\sigma w}\B(H)$
and $\Max_1\C(H)$ (see \cite{EgKr08}).

(6) Let $S,S'$ be sup-lattices and $f:S\to S',g:S'\to S$
Galois connections, i.e. $f(x)\leq y\LRa x\geq g(y)$
for every $x\in S,y\in S'$.
In that case we write $x\bot y$ and
put $xy =0$, or $1$ otherwise.
We have obtained a triad $(S,\Two ,S')$. 
\emph{Galois quantale}
$Q=\{ (\alpha,\beta)\in\Q(S)\otimes\Q(S')\mid
\alpha(x)\bot y\LRa x\bot\beta (y)\} $
constructed in \cite{Re04}
is a solution of the triad.
Notice that (4) is a special case for
duality $f:S\to S\op,g:S\op\to S$.

(7) For a quantum frame $F$ (see \cite{Rosi89})
we consider triad $(F,\tilde{F},F)$
where $\tilde{F}$ is a frame of two-sided
elements of $F$, and actions are defined by
\begin{align*}
xz=zx &=x\wedge z, & \text{(RT, TL)} & &
xy &= x\circ y & \text{(LR)}
\end{align*}
for $x,y\in F, z\in\tilde{F}$.
As elements of a quantum frame represent
q-open sets in quantized topology (see \cite{GiKu71}),
a solution of the triad provides a ``dynamical logic
of quantised topology''.
In contrast to other candidates like $\Q(F)$,
solutions respect the underlying
classical topology represented by
two-sided elements (central q-open sets).

(8) A special instance of (7) (and generalization
of (5)) is a complete orthomodular lattice $M$.
The quantum frame structure defined by J.~Rosick\'y
yields a triad $(M,Z(M),M)$ where $Z(M)$ is
a center of $M$.
Notice that $x\circ y$ can be calculated also
as $|x\,\dot{\wedge}\, y|$, i.e. a central cover of skew meet
(also known as Sasaki projection $\phi_x(y)=x\,\dot{\wedge}\, y=
(x\vee y\comp)\wedge y$), 
and $x\,\dot{\wedge}\, y$ coincide with $x\wedge y$
whenever $x$ or $y$ is central.
This suggests further examples emerging from
skew operations.
\end{example}

Let us recall that $R\otimes L$ is calculated as
a free sup-lattice on $R\times L$ factorized by
congruence generated by relations
\begin{align*}
\bigvee (r_i,l) &\sim \left(\bigvee r_i,l\right), &
\bigvee (r,l_i) &\sim \left(r,\bigvee l_i\right)
\end{align*}
for all $r,r_i\in R,l,l_i\in L$.
Every element of $R\otimes L$ representable by some $(r,l)$
is called a \emph{pure tensor}
and denoted by $r\otimes l$.

\begin{lemma}
Let $(L,T,R)$ be an triad in $\Sup$.
Let $R\otimes_TL$ be a sup-lattice quotient of $R\otimes L$ via
$(rt)\otimes l\sim r\otimes (tl)$ assumed for any
$r\in R,l\in L,t\in T$.
Let $Q_0$ be defined as $R\otimes_TL$ with operations
\emph{
\begin{align*}
(r\otimes l)r' &= r(lr'), & \text{(QR)} & &
l(r\otimes l') &= (lr)l', & \text{(LQ)} \\
(r\otimes l)(r'\otimes l') &= (r(lr'))\otimes l', & \text{(QQ)} & &
rl &= r\otimes l & \text{(RL)}
\end{align*}
}
for $r,r'\in R,l,l'\in L$.
Then $Q_0$ is a solution of $(L,T,R)$.
\end{lemma}
\begin{proof}
Since pure tensors are generators of $R\otimes L$
which is ``bifree'' on $R\times L$,
the assignments extend to all elements in a unique way.
Correctness follows from definition of $R\otimes_TL$.
All the associative laws can be proved only
for pure tensors and the proof is straightforward.
\end{proof}

\begin{lemma}
Let $(L,T,R)$ be an triad in $\Sup$.
Put
$$Q_1=\{ (\alpha,\beta)\in T\LMod (L,L)\times \RMod T(R,R)
\mid \alpha(l)r=l\beta(r)\}$$
and
\emph{
\begin{align*}
(\alpha,\beta)r &= \beta(r), &\text{(QR)} &
& l(\alpha,\beta) &= \alpha(l), &\text{(LQ)} \\
(\alpha,\beta)(\alpha',\beta') &= (\alpha'\alpha,\beta\beta'),
& \text{(QQ)} &
& rl &= ((-r)l,r(l-)). & \text{(RL)}
\end{align*}
}
Then $Q_1$ is a solution of $(L,T,R)$.
\end{lemma}
\begin{proof}
For every $l\in L,r\in R,(\alpha,\beta),(\alpha',\beta')\in Q_1$
we have $(\alpha'\alpha)(l)r=\alpha(l)\beta'(r)=(\beta\beta')(r)$,
hence $(\alpha,\beta)(\alpha',\beta')\in Q_1$.
The associative law (QQQ) evidently holds and $R,L$ are $Q_1$-modules
(QQR, LQQ).
Since elements of $Q_1$ are formed by $T$-module morphisms,
$R,L$ are also bimodules (QRT,TLQ).
Elements of the form $rl$ belongs to $Q_1$ thanks
to (TLR, TTL, LRT, RTT), and consequently (RLR, LRL, RTL) hold.
(LQR) follows by definition: $(l(\alpha,\beta))r=\alpha(l)r=l\beta(r)=
l((\alpha,\beta)r)$.
Finally, $((\alpha,\beta)r)l=\beta(r)l=((-\beta(r))l,\beta(r)(l-))=
((\alpha(-)r)l,\beta(r(l-)))=(\alpha,\beta)(rl)$ gives (QRL)
and similarly we would prove (RLQ).
\end{proof}

\begin{defin}
Let $C\stackrel{\phi}{\to}Q$ be a couple.
A quantale $K$ together with quantale morphisms
$\phi_0:C\to K,\phi_1:K\to Q$ such that $\phi_1\phi_0=\phi$
is called a \emph{couple factorization} if
the $K$-bimodule structure on $C$ obtained by
restricting scalars along $\phi_1$ makes
$\phi_0$ a coupling map. 
Namely, it holds that
\begin{align*}
\phi_0(\phi_1(k)c) &= k\phi_0(c), &
\phi_0(c\phi_1(k)) &= \phi_0(c)k
\end{align*}
for all $c\in C,k\in K$.
\end{defin}

\begin{theorem}\label{sol}
Let $Q$ be a solution of triad $(L,T,R)$.
The assignment $\phi(r\otimes l)=((-r)l,r(l-))$
determines a unital couple $Q_0\stackrel{\phi}{\to}Q_1$,
and $Q$ is a solution of the triad iff
there is a couple factorization
$Q_0\stackrel{\phi_0}{\to}Q\stackrel{\phi_1}{\to}Q_1$.
\end{theorem}
\begin{proof}
From (RTL, LRL, TLR, RLR, LRT) it follows that
whenever $Q$ is a solution
then $\phi_0:Q_0\to Q$ given by
$\phi_0(r\otimes l)=rl$ is a correctly defined
quantale morphism and together with (LQQ, QQR, QRL, RLQ) it
determines a couple with actions
$q(r\otimes l)=(qr)\otimes l$ and
$(r\otimes l)q=r\otimes (lq)$.
In particular $\phi:Q_0\to Q_1$ is a couple
and it is unital since $(id_L,id_R)\in Q_1$.
Further, if $Q$ is a solution then
(LQQ, QQR, TLQ, QRT, LQR) yield that
$\phi_1(q)=(-q,q-)$ defines a quantale
morphism $\phi_1:Q\to Q_1$.
Clearly, $\phi=\phi_1\phi_0$ and
$\phi_0(\phi_1(q)(r\otimes l))=\phi_0((\phi_1(q)r)\otimes l)=
(qr)l=q(rl)=q\phi_0(r\otimes l)$ and similarly
$\phi_0((r\otimes l)\phi_1(q))=\phi_0(r\otimes l)q$,
hence $\phi=\phi_1\phi_0$ is a couple factorization.

Conversely, for a couple factorization $\phi=\phi_1\phi_0$
we put
\begin{align*}
lq &= l\phi_1(q), & \text{(LQ)} & &
qr &= \phi_1(q)r, & \text{(QR)} \\
rl &= \phi_0(r\otimes l). & \text{(RL)}
\end{align*}
Then (TLQ, QRT, LQR) follow immediatelly,
(LQQ, QQR) hold since $\phi_1$ is a quantale morphism,
(LRL, RLR, RTL) since $\phi_0$ is a coupling map,
and (QRL, RLQ) since $\phi=\phi_1\phi_0$
is a couple factorization.
\end{proof}

\begin{defin}
A triad $(L,T,R)$ is called

\emph{strong} if $l\leq (l1_R)1_L$ and $r\leq 1_R(1_Lr)$
for every $l\in L,r\in R$,

\emph{unital} if $T$ is a unital quantale and
$L,R$ are unital $T$-modules, and

\emph{strict} if it is strong, unital, and $1_L1_R=e_T$.
\end{defin}

\begin{remark}
The triad in Example \ref{ex} (1) is strict
iff $Q$ is semiunital.
Examples (2 -- 8) provide strict triads.
\end{remark}

\begin{prop}\label{str}
(1) $\phi:Q_0\to Q_1$ is strong iff
$(L,T,R)$ is strong.

(2) If $(L,T,R)$ is strict,
then $\LL(Q_0)\cong\LL(Q_1)\cong L,
\R(Q_0)\cong\R(Q_1)\cong R$ (as modules over any solution)
and $\T(Q_0)\cong\T(Q_1)\cong T$ (as quantales).
In particular, $T$ is strictly two-sided.
\end{prop}
\begin{proof}
(1) If $\phi$ is strong, then $r=e_{Q_1}r\leq 1_{Q_1}r=
\phi(1_{Q_0})r=1_{Q_0}r=(1_R\otimes 1_L)r=1_R1_Lr$
and similarly for $l\in L$.

Conversely, let $(L,T,R)$ be strong.
Then for $r,r'\in R$ we have
$r\leq 1_R1_Lr'\Ra 1_R1_Lr\leq (1_R1_L)^2r'=1_R1_Lr'$.
Since $1_{Q_1}$ acts as a $T$-module endomorphism,
$1_{Q_1}(1_R1_Lr)=1_{Q_1}(1_R)1_Lr=1_R1_Lr$.
We have $1_{Q_1}r=\bigwedge_{r\leq 1_R1_Lr'}1_R1_Lr'=1_R1_Lr=
(1_R\otimes 1_L)r=1_{Q_0}r$.
In a similar way we prove that $l1_{Q_1}=l1_{Q_0}$,
hence $\phi(1_{Q_0})=1_{Q_1}$. 

(2) Due to \cite{EgKr08}, $\phi$ is an
isomorphism on right- and left-sided
elements.
Assume that $\rho\in\R(Q_1)$ and put $r=\rho 1_R$.
Then $\rho =\rho 1_{Q_1}$ yields that
$\rho$ acts as $\rho 1_{Q_1}=\rho 1_R1_L=r1_L$
on both $L$ and $R$.
Since $(L,T,R)$ is strict, we can recover
$r=re_T=r1_L1_R$.
On the other hand,
$(-r1_L,r1_L-)=\phi(r\otimes 1_L)\in Q_1$
is right-sided because $r1_L1_R1_L=r1_L$.
Similarly we check left-sided elements.

Every element $t\in T$ can be associated with
$(-1_Rt1_L,1_Rt1_L-)$ which is clearly
both right- and left-sided in $Q_1$
and from $1_Rt1_L$ can be recovered as
$t=1_L1_Rt1_L1_R$.
Conversely, every $\tau\in\T(Q)$ yields
$t=1_L\tau 1_R$ and then acts as $1_Rt1_L$.
For $t,t'\in T$ we have $1_Rt1_L1_Rt'1_L=
1_Rtt'1_L$, hence we have obtained
a quantale isomorphism. 
\end{proof}

\begin{remark}
Recall that a quantale $Q$ is \emph{faithful} \cite{PeRo97}
if $\forall l\in\LL(Q),r\in\R(Q)\ lq=lq'$ and $qr=q'r$
implies that $q=q'$ for every $q,q'\in Q$. 

In the case of a strict triad we have obtained that
$Q_0$ is generated by its right- and left-sided elements
while $Q_1$ is faithful.
\end{remark}

\section{Central, involutive, and Girard triads}

\begin{defin} 
A triad $(L,T,R)$ is called \emph{central}
if $LR$ is a subset of the center $Z(T)$ of $T$,
i.e. $lrt=tlr$ for every $l\in L,t\in T,r\in R$.
\end{defin}

The following statement exhibits analogs
of an embedding of a center and of a trace
in operator algebras.

\begin{prop}
Let $(L,T,R)$ be a central triad.
Then assignment $t\mapsto (t-,-t)$ defines
a quantale morphism $\zeta:T\to Q_1$
and assignment $r\otimes l\mapsto lr$
a sup-lattice morphism $\tau:Q_0\to T$.
Elements of $\zeta(T)$ are central in $Q_1$
and elements of $\tau\adj(T)$ are cyclic
for any couple of solutions
$Q_0\stackrel{\phi_0}{\to}Q$.
\end{prop}
\begin{proof}
For every $t\in T$ and $r\in R, l\in L$ we have
$tlr=lrt$, hence $\zeta(t)\in Q_1$.
For $t,t'\in T$ we get
$\zeta(t)\zeta(t')=((t-)\circ (t'-),(-t')\circ (-t)=
(tt'-,-tt')=\zeta(tt')$.
Finally, for $(\alpha,\beta)\in Q_1$ we have
$\zeta(t)(\alpha,\beta)=(\alpha(t-),\beta(-)t)=
(t\alpha(-),\beta(-t))=(\alpha,\beta)\zeta(t)$.

Since $lr$ is central for every $l\in L,r\in R$,
$\tau(rt\otimes l)=lrt=tlr=\tau(r\otimes tl)$
yields correctness of $\tau$.
For $q\in Q$ we have $\tau(q(r\otimes l))=
\tau((qr)\otimes l)=lqr\leq t\LRa
\tau((r\otimes l)q)\leq t$, i.e. $\tau\adj(t)$
is cyclic.
\end{proof}

\begin{defin}
An triad $(L,T,R)$ is called \emph{involutive}
if $T$ is involutive and
there is an isomorphism $L\stackrel{*}\cong R$
such that 
\begin{align*}
(tl)^* &= l^*t^*, & (rt)^* &= t^*r^*, \\
(lr)^* &= r^*l^* & &
\end{align*}
for all $r\in R,l\in L,t\in T$.

A solution $Q$ of involutive triad $(L,T,R)$
is called \emph{involutive} if
$Q$ is an involutive quantale and
\begin{align*}
(qr)^* &= r^*q^*, & (lq)^* &= q^*l^*, \\
(rl)^* &= l^*r^* & &
\end{align*}
for all $q\in Q, r\in R, l\in L$.

A couple $C\stackrel{\phi}{\to}Q$ is
said to be \emph{involutive}
if $\phi$ is an involutive morphism of involutive
quantales and $(qc)^*=c^*q^*$ for every $q\in Q,c\in C$.
A couple factorization $C\stackrel{\phi_0}{\to}K
\stackrel{\phi_1}{\to}Q$ is called \emph{involutive}
if $K$ is an involutive quantale and $\phi_0,\phi_1$
are involutive morphisms.
\end{defin}

\begin{theorem}
If $(L,T,R)$ is an involutive triad, then
$Q_0\stackrel{\phi}{\to}Q_1$ is an involutive couple.
Involutive solutions correspond to involutive
couple factorizations of $\phi$.
\end{theorem}
\begin{proof}
On $Q_0$ put $(r\otimes l)^*=l^*\otimes r^*$.
On $Q_1$ put $(\alpha,\beta)^*=(\bar{\beta},\bar{\alpha})$
where $\bar{\alpha}(r)=\alpha(r^*)^*$ and
$\bar{\beta}(l)=\beta(l^*)^*$.
The rest follows straightforwards.
\end{proof}

A triad $(L,T,R)$ is said to be \emph{Girard}
if $T$ admits a cyclic element $d_T$
such that assignment
$$ r\leq l\comp (\LRa l\leq r\comp)\LRa lr\leq d_T $$
provides a duality $L\op\stackrel{\bot}{\cong}R$.

\begin{theorem}\label{gir}
If $(L,T,R)$ is a Girard triad,
then $Q_1\cong T\LMod (L,L)\cong\RMod T(R,R)$
and $\phi:Q_0\to Q_1$ is a Girard couple.
\end{theorem}
\begin{proof}
Let $(\alpha,\beta)\in Q_1$.
Then $\alpha(r)\leq r'\LRa
(r')\comp\alpha(r)=\beta((r')\comp )r\leq d_T
\LRa r\leq\beta((r')\comp)\comp$.
We have proved that $\beta=\alpha^*$ and
hence the assertion.

Put $d_Q=\bigvee_{lr\leq d_T}r\otimes l$.
For $l\in L,t\in T,r\in R$ we have
$lrt\leq d_T\LRa tlr\leq d_T$ by cyclicity,
hence $r\otimes l\leq d_Q\LRa lr\leq d_T$,
regardless how the pure tensor $r\otimes l$
is represented in $R\otimes_TL$.

Now put $\pi^{l'}_l=\bigvee\{ \alpha\in T\LMod (L,L)\mid
\alpha(l)\leq l'\}$
and observe that
$$ \alpha =\bigwedge_{\alpha(l)\leq l'}\pi^{l'}_l $$
for every $\alpha\in T\LMod (L,L)$.
Then $\alpha(tl)\leq r\comp\LRa
\alpha(tl)r=t\alpha(l)r\leq d_T\LRa
\alpha(l)(rt)\leq d_T\LRa \alpha(l)\leq (rt)\comp$
yields that $\pi^{(rt)\comp}_l=\pi^{r\comp}_{tl}$
and thus pure tensors of $R\otimes_TL$ are
duals of $\pi$s.
Since elements of $T\LMod (L,L)$ preserve
arbitrary suprema, we deduce that they are,
as the meets of $\pi$s, organized dually
to joins of pure tensors and hence
$Q_1\cong T\LMod (L,L)$ is dual to $Q_0$.
(This fact was mentioned by A.~Joyal and M.~Tierney
\cite{JoTi84} for commutative $T$.)
Finally, $(r\otimes l)(\alpha,\alpha^*)=
r\otimes\alpha(l)\leq d_Q\LRa
\alpha(l)r\leq d_T\LRa
\alpha(l)\leq r\comp\LRa
\alpha\leq\pi^{r\comp}_l=(r\otimes l)\comp$
establishes the Girard duality.
\end{proof}

Recall from \cite{KrPa02} that a quantale
$Q$ is said to be \emph{strictly faithful}
if $(\forall l\in\LL(Q),r\in\R(Q)\ lqr=lq'r)\Ra q=q'$
for all $q,q'\in Q$ and from \cite{Kr03} that it is
called \emph{distributive} if
$(r\vee q)\wedge(l\vee q)=rl\vee q$ for
all $q,q'\in Q, r\in\R(Q), l\in\LL(Q)$.

\begin{prop}
Let $(L,T,R)$ be a strict Girard triad.
Then $Q_1$ is strictly faithful
and if $T$ is distributive, then $Q_1$ is distributive.
\end{prop}
\begin{proof}
By \ref{str} we can identify $L$ with $\LL(Q)$
and $R$ with $\R(Q)$.
If $r\ngeq r'$ then $r\comp r\leq d_T$
but $r\comp r'\nleq d_T$, hence
$(\forall l\in L\ lr=lr')\Ra r=r'$ for every $r,r'\in R$.
Thus from $lqr=lq'r$ for all $r\in R,l\in L$
we derive $qr=q'r$ for all $r\in R$.
By a similar argument one can also derive $lq=lq'$
for all $l\in L$.
We have that $q$ and $q'$ are not
distinguished either on $L$ or on $R$
and thus they are equal. 

Assume now that $T$ is distributive and
$l'(rl\vee q)r'\leq t$
for $l,l'\in L,r,r'\in R,q\in Q$,
thus $l'rlr'\leq t$ and $l'qr'\leq t$.
Then $l'((r\vee q)\wedge(l\vee q))r'\leq
(l'rr'\vee l'qr')\wedge(l'lr'\vee l'qr')\leq
(l'r\vee t)\wedge(lr'\vee t)=l'rlr'\vee t=t$.
From strict faithfulness we obtain
$(r\vee q)\wedge(l\vee q)\leq rl\vee q$.
The converse inequality always holds.
\end{proof}

\begin{remark}
Since $T$ is a frame in examples (3 -- 8) of \ref{ex},
the triads are central and solutions $Q_1$
are distributive.

Involutive triads arise from involutive rings
(in particular C*-algebras), self-dual sup-lattices,
symmetric Galois connections, and quantum frames.

Triad $(S\op,\Two,S)$ is Girard for every sup-lattice
$S$ and $\phi:Q_0\to Q_1$ is the Girard couple studied
in \cite{EgKr08}.
More generally, the triad $(M,Z(M),M)$ from \ref{ex} (8)
is Girard.

Any W*-algebra provides a Girard triad of
ideals closed in normal topology.
The quantale $Q_1$ obtained from non-atomistic
W*-algebra is distributive but non-spatial
(because it does not have enough maximal right-sided
elements) and represents a natural non-commutative analogy
of a pointfree locale. 
W*-algebras with a non-trivial center produce examples of
strictly faithful quantales (with idempotent right- and left-sided
elements)
which are not \emph{simple} (see \cite{PeRo97}). 
\end{remark}

\bibliography{david}
\bibliographystyle{acm}

\end{document}